\documentclass{article}

\usepackage{amsthm,amsmath,amssymb,hyperref,graphicx,subfig,pifont,geometry,fleqn,hyperref}

\usepackage{color}
\usepackage{relsize}

\newcommand\nextedge[1]{#1^{+}}

\newcommand\outdegree[2]{deg_#1^{+}(#2)}
\newcommand\nsptrees[2]{\mathcal{T}_{#1}({#2})}

\newcommand\chipadd[1]{E_{#1}}
\newcommand\nsp[2]{\mathcal{T}_{#1}(#2)}

\newcommand\outdeg[2]{deg_{#1}^{+}(#2)}
\newcommand\indegree[2]{deg_{#1}^{-}(#2)}
\newcommand\indeg[2]{deg_{#1}^{-}(#2)}
\newcommand\numedges[3]{a_G({#2,#3})}
\def\crstate{single-chip-and-rotor }
\newtheorem{theorem}{Theorem}
\newtheorem{lemma}{Lemma}
\newtheorem{proposition}{Proposition}
\newtheorem{corollary}{Corollary}
\newtheorem*{ques}{Question}

\usepackage{graphicx}
\begin{document}

\title{Orbits of rotor-router operation and stationary distribution of random walks on directed graphs\thanks{Paper was partially sponsored by Vietnam Institute for Advanced Study in Mathematics (VIASM), the Vietnamese National Foundation for Science and Technology Development (NAFOSTED), and the European Research Council under the European Community's Seventh Framework Programme (FP7/2007-2013 Grant Agreement no. 257039.}}
%\
%tnoteref{t}
%}
%\tnotetext[t]{Paper was partially sponsored by Vietnam Institute for Advanced Study in Mathematics (VIASM), the Vietnamese National Foundation for Science and Technology Development (NAFOSTED), and the European Research Council under the European Community's Seventh Framework Programme (FP7/2007-2013 Grant Agreement no. 257039.}         % Enter your title between curly braces

\author{Trung Van Pham} 
%\ead{pvtrung@math.ac.vn}  
%\address{Institute of Mathematics, VAST\\
%Department of Mathematics for Computer Science\\
%18 Hoang Quoc Viet Road, Cau Giay District, Hanoi, Vietnam.\\
%Technische Universit\"{a}t Dresden\\
%Fachrichtung Mathematik, Institut f\"{u}r Algebra\\
%01062 Dresden, Germany
%} 
%\begin{keyword}
%Eulerian walker \sep oriented spanning tree  \sep random walk \sep recurrent state\sep rotor-router model  \sep spanning tree\sep stationary distribution. 
%\end{keyword}   
%\ead{pvtrung@math.ac.vn}
%\address{Institute of Mathematics, VAST\\
%18 Hoang Quoc Viet Road, Cau Giay District, Hanoi, Vietnam
%}
%\author{}
%\address{Department of Mathematics of Computer Science\\
%Institute of Mathematics\\
%18 Hoang Quoc Viet Road, Cau Giay District, Hanoi, Vietnam}
%\email{pvtrung@math.ac.vn}
\date{}          % Enter your date or \today between curly braces

\maketitle
\begin{abstract}
The rotor-router model is a popular deterministic analogue of random walk. In this paper we prove that all orbits of the rotor-router operation have the same size on a strongly connected directed graph (digraph) and give a formula for the size. By using this formula we address the following open question about orbits of the rotor-router operation: Is there an infinite family of non-Eulerian strongly connected digraphs such that the rotor-router operation on each digraph has a single orbit? 

It turns out that on a strongly connected digraph the stationary distribution of the simple random walk coincides with the frequency of vertices in a rotor walk. In this common aspect a rotor walk simulates a random walk. This gives one similarity between two models on (finite) digraphs.
% We also study the random walk on the set of single-chip-and-rotor states which is induced by the random walk on a strongly connected digraph. We show that its stationary distribution is unique and uniform on the set of recurrent states. This means that recurrent states occur at the same almost sure frequency when the chip performs a random walk.
% Moreover, we show that when the chip performs a random walk, recurrent states occurs with the same frequency almost surely.
\end{abstract}

\section{Introduction}
The rotor-router model is a popular deterministic analogue of random walk that was discovered firstly by  Priezzhev, D. Dhar et al.  as a model of self organized criticality  under the name ``Eulerian walkers" \cite{PDDS96}. The model has become popular recently because it shows many surprising properties which are similar to those of random walk  \cite{CDST06,CS06,DF09,HP10}. The model was studied mostly on $\mathbb{Z}^d$ with the problems similar to those of the random walk. Although the model was defined firstly on (finite) graphs, there are not many known results on this class of graphs, in particular a similarity between the two models on digraphs is still unknown.
% In this paper we study orbits of its operation on finite digraphs.  In particular we prove all orbits of the operation have the same size and present a formula for the size. By giving an explicit formula for the stationary distribution of the random walk on a finite digraph we present a first similarity to the random walk on finite digraphs.
% The model is studied not only on finite graphs but also on infinite graphs.  In this paper we study the model on finite digraphs.

Let $G=(V,E)$ be a connected digraph. For each vertex $v$ the set of the edges emanating from $v$ is equipped with a cyclic ordering. We denote by $\nextedge{e}$ the next edge of edge $e$ in this order. A vertex $s$ of $G$ is called \emph{sink} if its outdegree is $0$. A \emph{rotor configuration} $\rho$ is a map  from the set of non-sink vertices of $G$ to $E$ such that for each non-sink vertex $v$ of $G$ $\rho(v)$ is an edge emanating from $v$. We start with a rotor configuration and a chip placed on some vertex of $G$. When a chip is at a non-sink vertex $v$, routing chip at $v$ with respect to a rotor configuration $\rho$ means the process of updating $\rho(v)$ to $\rho(v)^{+}$, and then the chip moves along the updated edge $\rho(v)$ to the head. The chip is now at the head of the edge $\rho(v)$. We define a \emph{\crstate  state} (often briefly \emph{state}) to be a pair $(v,\rho)$ of a vertex and a rotor configuration $\rho$ of $G$. The vertex $v$ in $(v,\rho)$ indicates the location of the chip in $G$. When $v$ is not a sink, by routing the chip at $v$ we obtain a new state $(v',\rho')$. This procedure is called \emph{rotor-router operation}. Look at \linebreak Figure \ref{fig:im000102} for an illustration of the rotor-router operation. In this example the acyclic ordering at each vertex is adapted to the counter-clockwise rotation. When the chip is at a sink, it stays at the sink  forever, and therefore the rotor-router operation fixes such states. A sequence of vertices of $G$ indicating the consecutive locations of the chip is called a \emph{rotor walk}.

\begin{figure}
\centering
%\subfloat[A grid graph]{\label{fig:im00}\includegraphics[bb=0 0 272 272,width=1.28in,height=1.33in]{im0.eps}}
\subfloat[A grid graph]{\label{fig:im00}\includegraphics[bb=0 0 272 272,width=1.28in,height=1.33in]{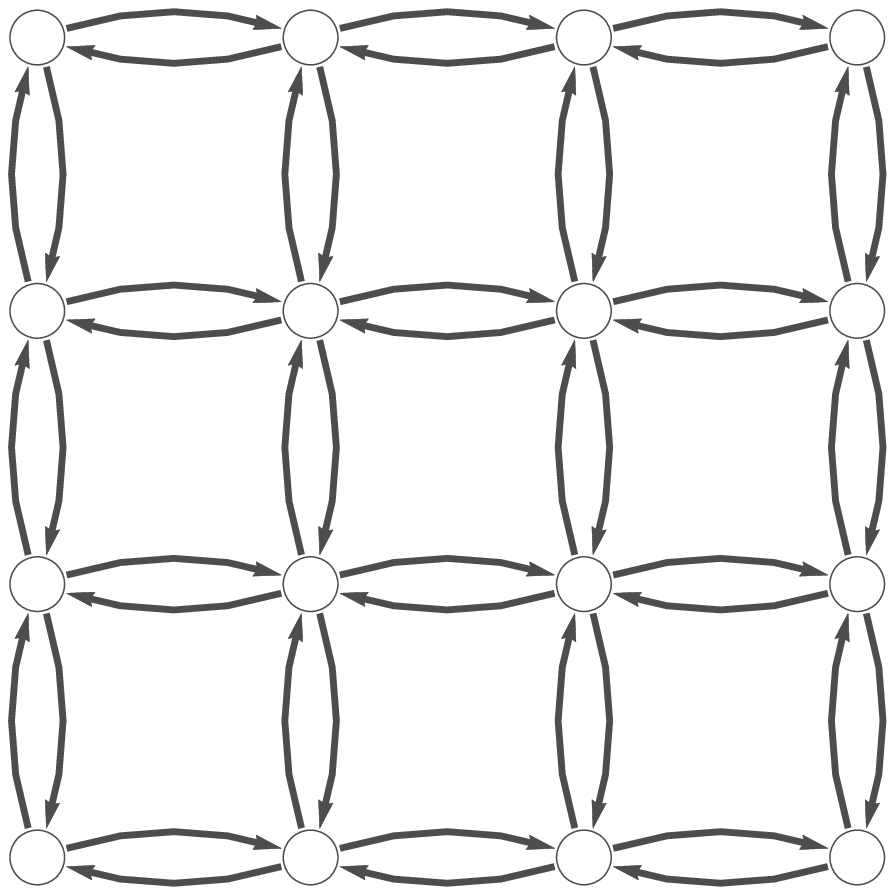}} \quad \quad
%\subfloat[A {\crstate} state (the plane edges for rotor configuration, and the black vertex indicates the location of the chip)]{\label{fig:im01}\includegraphics[bb=0 0 310 310,width=1.28in,height=1.33in]{im1.eps}}
\subfloat[A {\crstate} state (the plane edges for rotor configuration, and the black vertex indicates the location of the chip)]{\label{fig:im01}\includegraphics[bb=0 0 310 310,width=1.28in,height=1.33in]{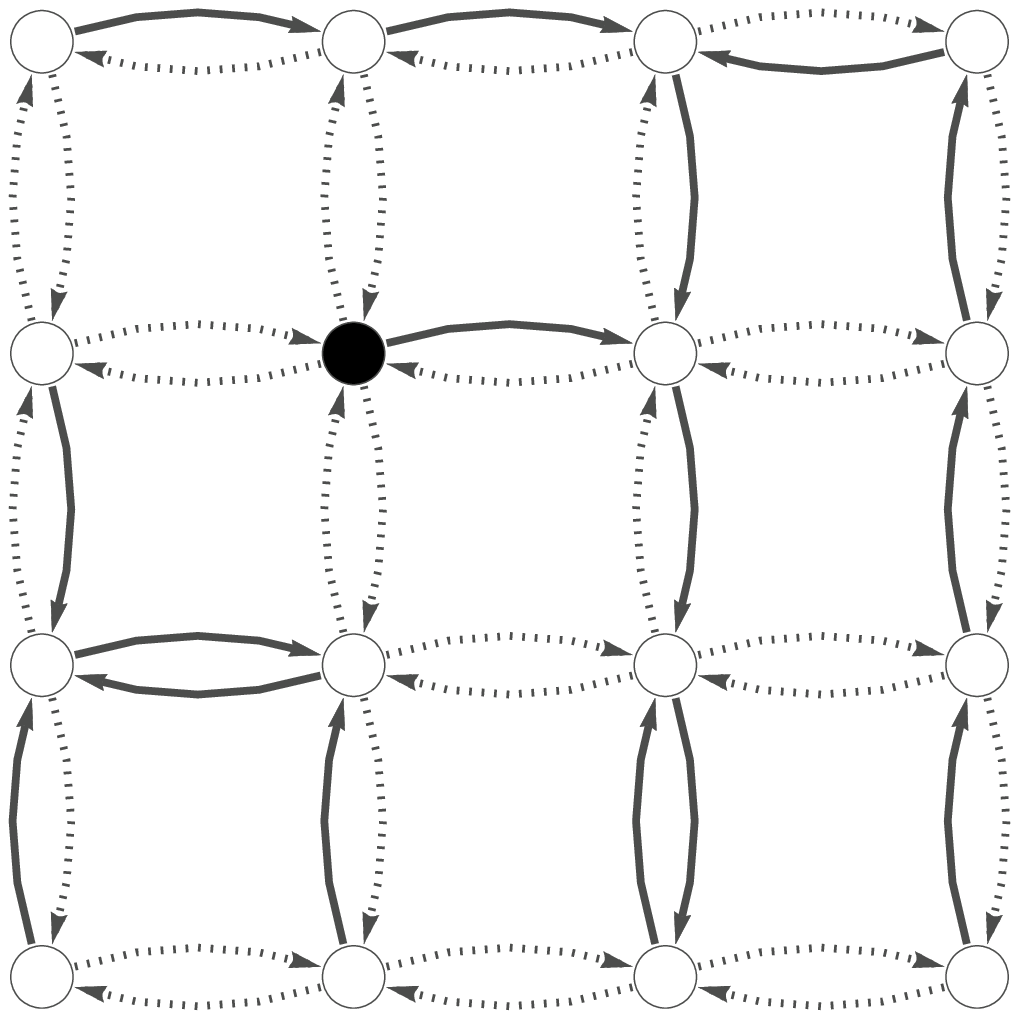}}\quad \quad
%\subfloat[Resulting {\crstate} state]{\label{fig:im02}\includegraphics[bb=0 0 286 285,width=1.28in,height=1.33in]{im2.eps}}
\subfloat[Resulting {\crstate} state]{\label{fig:im02}\includegraphics[bb=0 0 286 285,width=1.28in,height=1.33in]{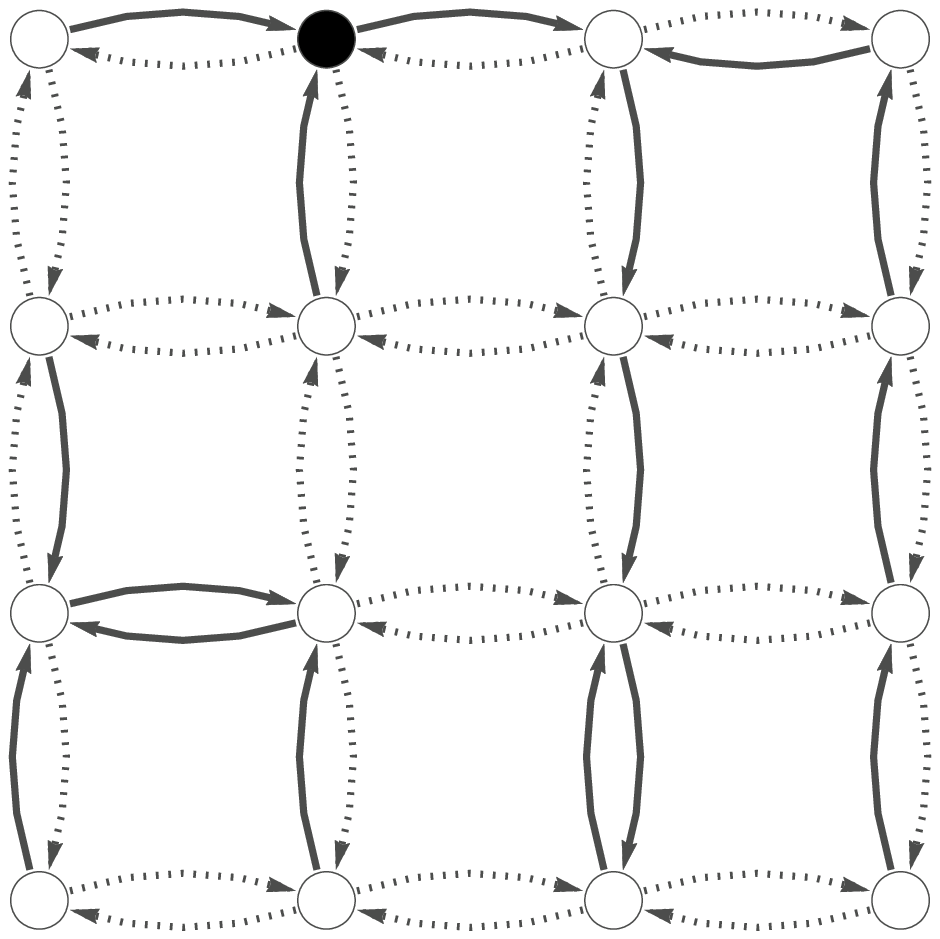}}
\caption{}
\label{fig:im000102}
\end{figure}

If $G$ has no sink, a state $(v,\rho)$ is \emph{recurrent} if starting from $(v,\rho)$ and after some steps (positive number of steps) of iterating the rotor-router operation we obtain $(v,\rho)$ again. The \emph{orbit} of a recurrent state is the set of all states which are reachable from the recurrent state by iterating the rotor-router operation.  Holroyd et al. gave a characterization for recurrent states \cite{HLMPPW08}.  By investigating orbits of recurrent states on an Eulerian digraph the authors observed that sizes of orbits are extremely short while number of recurrent states is typically exponential in number of vertices. They asked whether there is an infinite family of non-Eulerian strongly connected digraphs such that all recurrent states of each digraph in the family are in a single orbit (Question 6.5 in \cite{HLMPPW08}). An immediate fact from the results in \cite{HLMPPW08,PDDS96} is that all orbits have the same size on an Eulerian digraph, namely $|E|$. The following main theorem shows that this fact holds not only for Eulerian digraphs but also for strongly connected digraphs.
% So it is natural and important to ask whether this fact also holds for general digraphs. For this problem we have the following main result.

 \begin{theorem}\label{theo:orbitsize}
Let $G=(V,E)$ be a strongly connected digraph, and $c$ be a recurrent state of $G$.  Then the size of the orbit of $c$ is $\frac{1}{M}\underset{v\in V}{\mathlarger{\sum}} \outdegree{G}{v} \nsptrees{G}{v} $, where $\nsptrees{G}{v}$ denotes the number of oriented spanning trees of $G$ rooted at $v$ and $M$ denotes the greatest common divisor of the numbers in $\{ \nsptrees{G}{v}:v\in V \}$. As a corollary, the number of orbits is $M$.
\end{theorem}
Note that the value $\nsp{G}{v}$ can be computed efficiently by using the \emph{matrix-tree theorem} \cite{Sta99}. Thus one can compute the size of an orbit efficiently without listing all states in an orbit. Although the orbits depend on the choice of cyclic orderings, it is interesting that the size of orbits is independent of the choice of cyclic orderings. All recurrent states are in a single orbit if and only if $M=1$. By doing computer simulations on random digraph $G(n,p)$ with $p\in (0,1)$ fixed, we observe that $M_{n,p}=1$ occurs with a high frequency when $n$ is sufficiently large. This observation contrasts with the observation on Eulerian digraphs when one sees the orbits are extremely short  \cite{HLMPPW08,PDDS96}.
\begin{ques}
Let $p\in (0,1)$ be fixed. Is $\textbf{Pr}\{M_{n,p}=1\}\to 1$ as $n\to \infty$?
\end{ques}
\noindent By using Theorem \ref{theo:orbitsize} we give a positive answer for the open question of Holroyd et al. in \cite{HLMPPW08}.
\begin{theorem}
\label{theo:singleorbit}
There is an infinite family of non-Eulerian strongly connected digraphs $G_n$ such that for each $n$ all recurrent states of $G_n$ are in a single orbit.
\end{theorem}
\noindent Note that the recurrent states of a directed cycle graph with $n$ vertices are in a single orbit. This is the reason why the digraphs in the theorem are required to be non-Eulerian.

%For $G$ being a connected digraph such that $\outdegree{G}{v}\geq 1$ for any $v\in V$ \emph{the random walk} on $G$ is a process of moving the chip on $V$ for which the chip at a vertex $v$ chooses an edge $e$ emanating from $v$ at random, and then moves to the head of $e$. This process is a Markov chain on $V$. A random sequence of vertices of $G$ indicating the consecutive locations of the chip in this process is called a \emph{random walk}. The \emph{stationary distribution} $\pi$ on $V$ is an important characteristic which can be thought of as almost sure frequency of vertices in a random walk. If $G$ is strongly connected, the so-called Markov chain tree theorem states that the stationary distribution $\pi$ of $G$ is given by  $\pi(v)=\frac{\nsp{G}{v} \outdegree{G}{v}}{\underset{w\in V}{\mathlarger{\sum}}\nsp{G}{w}\outdegree{G}{w}}$ for any $v\in V$ \cite{AT89,LP95}. Let $(X_0,X_1,X_2,\dots)$ be a random walk. It follows from the ergodic theorem that $\textbf{Pr}\left\{\underset{t\to \infty}{\lim}\frac{\underset{0\leq i\leq t-1}{\mathlarger{\sum}}\mathlarger{\textbf{1}}_{\{X_i=v\}}}{t}=\pi(v)\right\}=1$ for any $v\in V$, where $\mathlarger{\textbf{1}}_{A}$ denotes the indicator function, for which $\textbf{1}_A(x)=1$ if $x\in A$, and $\textbf{1}_A(x)=0$ otherwise \cite{LPW09}.

For $G$ being strongly connected let $(v_i)_{i=0}^{\infty}$ be a rotor walk. As we will show in the proof of Theorem \ref{theo:orbitsize} the number of occurences of the chip at a vertex $v$ in an orbit is $\frac{1}{M}\nsp{G}{v}\outdegree{G}{v}$. This implies that in a rotor walk the chip visits a vertex $v$ with the frequency $\underset{t\to \infty}{\lim}\frac{\underset{0\leq i\leq t-1}{\mathlarger{\sum}}\mathlarger{\textbf{1}}_{\{v_i=v\}}}{t}=\frac{\nsp{G}{v}\outdegree{G}{v}}{\underset{w\in V}{\mathlarger{\sum}}\nsp{G}{w}\outdegree{G}{w}}$. This frequency coincides with the stationary distribution of the simple random walk on $G$. Thus a rotor walk simulates a random walk in this aspect. It would be interesting to explore more properties of random walks by investigating properties of rotor walks on finite digraphs.

The structure of this paper is as follows. In Section \ref{backgrounds} we will give some background on the rotor-router model. The definitions and the results on the rotor-router model we present in this section are mainly from \cite{HLMPPW08}. In Section \ref{section:orbitsize} we will give a proof for Theorem \ref{theo:orbitsize} and use this result to give a proof for Theorem \ref{theo:singleorbit}. 
\section{Background on rotor-router model}
\label{backgrounds}
In this paper all digraphs may have loops and multi-edges. For a digraph $G$ we denote by $V(G)$ and $E(G)$ the set of vertices and the set of edges of $G$, respectively. In this section we work with a digraph $G=(V,E)$. The outdegree (resp. indegree) of a vertex $v$ is denoted by $\outdegree{G}{v}$ (resp. $\indegree{G}{v}$). For two distinct vertices $v$ and $v'$ we denote by $\numedges{G}{v}{v'}$ the number of edges connecting $v$ to $v'$. Note that $\numedges{G}{v}{v}$ is the number of loops at $v$. A \emph{walk} in $G$ is an alternating sequence of vertices and edges $v_0,e_0,v_1,e_1,\dots,v_{k-1},e_{k-1},v_k$ such that for each $i\leq k-1$ we have $v_i$ and $v_{i+1}$ are the tail and the head of $e_i$, respectively. A \emph{path} is a walk in which all vertices are distinct. For simplicity we often represent a walk (or path) by $e_0,e_1,\dots,e_{k-1}$, or $v_0,v_1,v_2,\dots,v_k$ if there is no danger of confusion. A subgraph $T$ of $G$ is called \emph{oriented spanning tree} of $G$ rooted at a vertex $s$ of $G$ if $s$ has outdegree $0$ in $T$ for every vertex $v$ of $G$ there is unique path from $v$ to $s$ in $T$. If $G$ has no sink, a \crstate state $(w,\rho)$ is called a \emph{unicycle} if the subgraph of $G$ induced by the edges in $\{\rho(v):v\in V\}$ contains a unicycle and $w$ lies on this cycle. Observe that the rotor-router operation takes unicycles to unicycles. Look at Figure \ref{fig:im030504} for examples of unicycles and non-unicycles.
\begin{figure}[!h]
\centering
%\subfloat[A unicycle]{\includegraphics[bb=0 0 272 272,width=1.4in,height=1.4in]{im3.eps}}
\subfloat[A unicycle]{\includegraphics[bb=0 0 272 272,width=1.4in,height=1.4in]{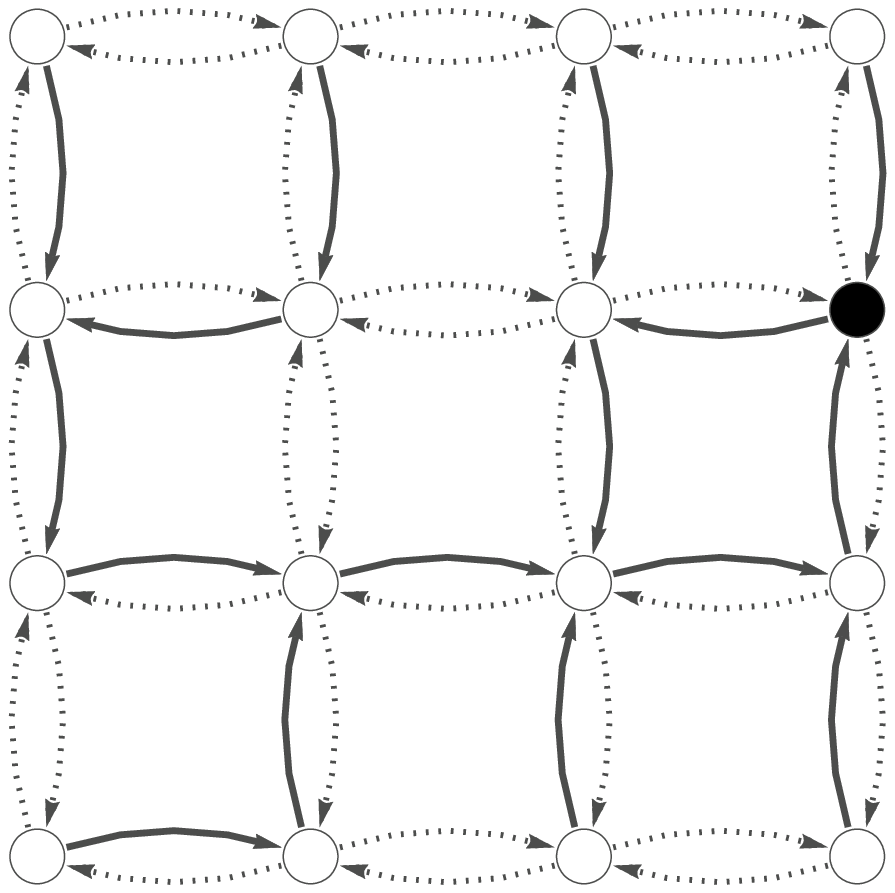}} \quad \quad
%\subfloat[A non-unicycle]{\includegraphics[bb=0 0 272 272,width=1.4in,height=1.4in]{im5.eps}}
\subfloat[A non-unicycle]{\includegraphics[bb=0 0 272 272,width=1.4in,height=1.4in]{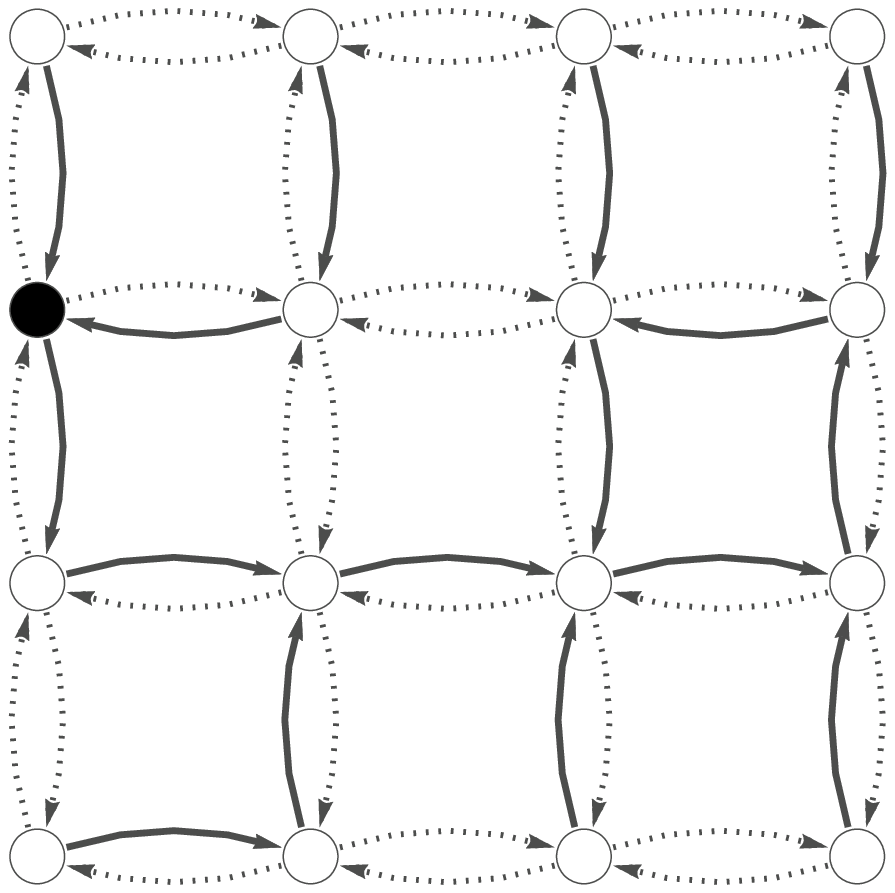}} \quad \quad
%\subfloat[A non-unicycle]{\includegraphics[bb=0 0 272 272,width=1.4in,height=1.4in]{im4.eps}}
\subfloat[A non-unicycle]{\includegraphics[bb=0 0 272 272,width=1.4in,height=1.4in]{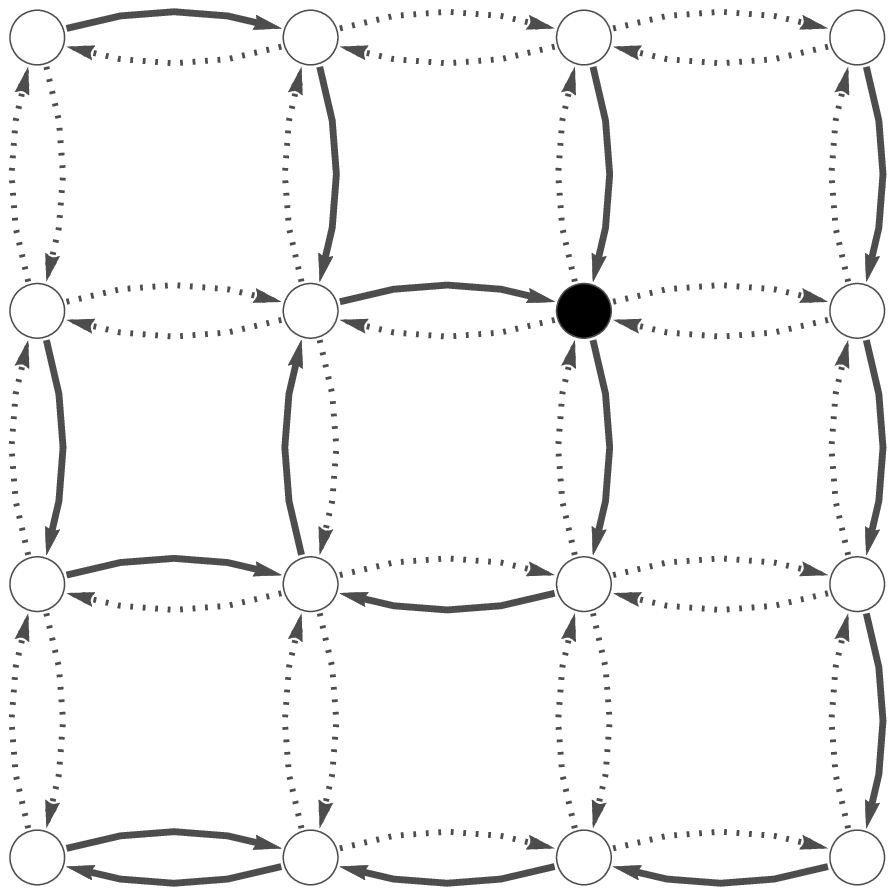}}
\caption{}
\label{fig:im030504}
\end{figure}
For a characterization of recurrent states we have the following lemma.
\begin{lemma}\cite{HLMPPW08}
Let $G=(V,E)$ be a strongly connected digraph. A state $(w,\rho)$ is recurrent if and only if $(w,\rho)$ is a unicycle.
\end{lemma}

Fix a linear order $v_1<v_2<\dots<v_n$ on $V$, where $n=|V|$. The $n\times n$ matrix given by 
$$
\Delta_{i,j}=\begin{cases}
-\numedges{G}{v_i}{v_j}&\text{if }i\neq j\\
\outdegree{G}{v_i}-\numedges{G}{v_i}{v_i}&\text{if } i=j,
\end{cases}
$$
is called the \emph{Laplacian} matrix of $G$. Let $j\in \{1,2,\dots,n\}$ be an arbitrary and $\Delta'$ be the matrix which is obtained from $\Delta$ by deleting the $j^{th}$ row and the $j^{th}$ column. We define the equivalence relation $\sim$ on $\mathbb{Z}^{n-1}$ by $c_1\sim c_2$ iff there is $z\in \mathbb{Z}^{n-1}$ such that $c_1-c_2=z \Delta'$. We recall the matrix-tree theorem.
\begin{theorem}\cite{Sta99}
The number of oriented spanning trees of $G$ rooted at $v_j$ is equal to the number of equivalence classes of $\sim$, and therefore equal to $Det(\Delta')$.
\end{theorem}
\noindent It follows from the theorem that the value $\nsp{G}{v}$ can be computed efficiently by using the Laplacian matrix.

A vertex $s$ of $G$ is called a \emph{global sink} of $G$ if $s$ has outdegree $0$ and for every vertex $v$ of $G$ there is a path from $v$ to $s$. If $G$ has a global sink $s$, a rotor configuration $\rho$ on $G$ is called \emph{acyclic} if the subgraph of $G$ induced by the edges in $\{ \rho(v): v\neq s \}$ is acyclic. Observe that if $\rho$ is acyclic, then $\{\rho(v): v\neq s\}$ is an oriented spanning tree of $G$ rooted at $s$. The \emph{chip-addition operator $\chipadd{v}$} is the procedure of adding one chip to a vertex $v$ of $G$ and routing this chip until it arrives at the sink. This procedure results the rotor configuration $\rho'$, and we write $\chipadd{v}\rho=\rho'$. Look at Figure \ref{fig:im0607} for an illustration of the chip-addition operator.
\begin{figure}
\centering
%\subfloat[A digraph with a global sink $s$]{\includegraphics[bb=0 0 201 200,width=1.40in,height=1.40in]{im8.eps}}
\subfloat[A digraph with a global sink $s$]{\includegraphics[bb=0 0 201 200,width=1.4in,height=1.4in]{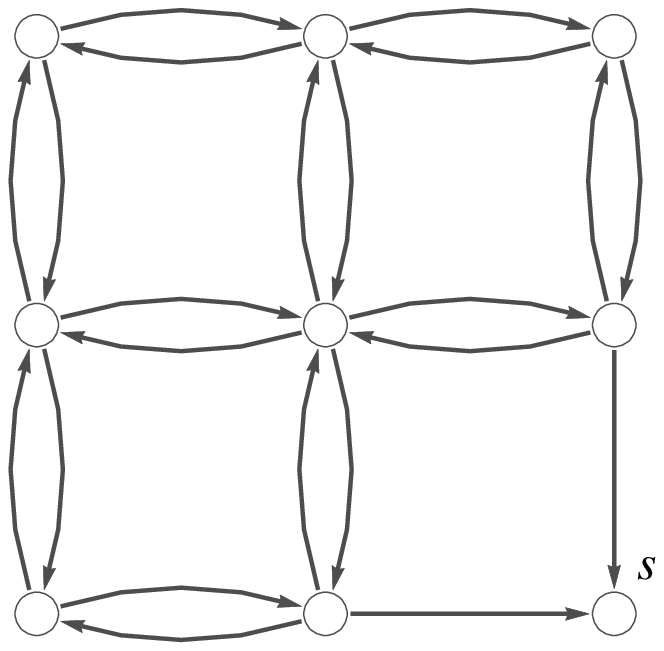}} \quad \quad
%\subfloat[A rotor configuration $\rho$ with a chip at vertex $v$]{\includegraphics[bb=0 0 201 206,width=1.40in,height=1.40in]{im6.eps}}
\subfloat[A rotor configuration $\rho$ with a chip at vertex $v$]{\includegraphics[bb=0 0 201 206,width=1.4in,height=1.4in]{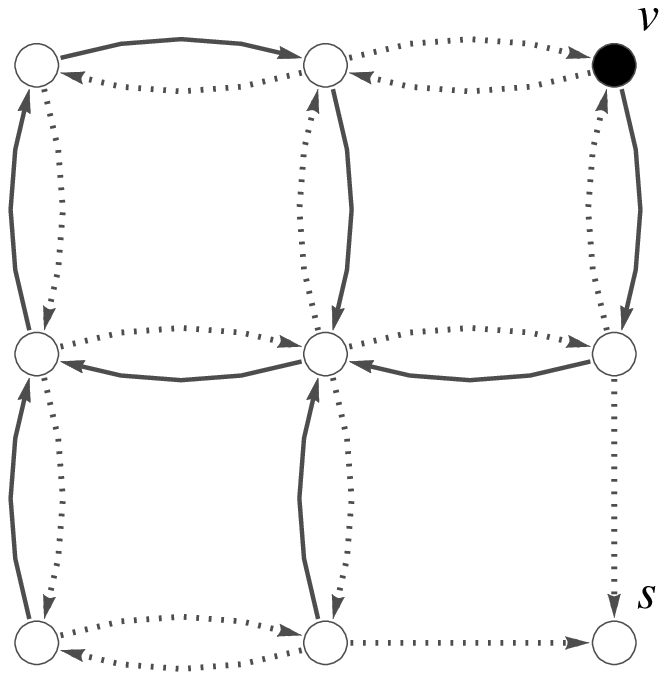}}  \quad \quad
%\subfloat[When the chip arrives at the sink: $\chipadd{v} \rho$ (plane edges)]{\includegraphics[bb=0 0 201 206,width=1.4in,height=1.4in]{im7.eps}}
\subfloat[When the chip arrives at the sink: $\chipadd{v} \rho$ (plane edges)]{\includegraphics[bb=0 0 201 206,width=1.4in,height=1.4in]{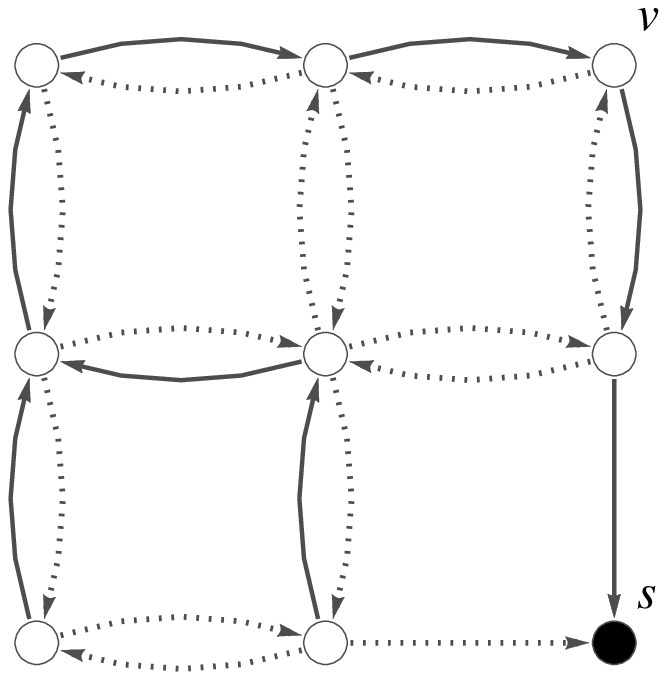}}
\caption{}
\label{fig:im0607}
\end{figure}
\begin{lemma}\cite{HLMPPW08}
Let $G=(V,E)$ be a digraph with a global sink $s$. Then the chip-addition operator is commutative. Moreover, for each $v\in V$ the operator $\chipadd{v}$ is a permutation on the set of acyclic rotor configurations of $G$.
\end{lemma}

If $G$ has a global sink $s$, a chip configuration on $G$ is a map from $V\backslash \{s\}$ to $\mathbb{N}$. The commutative property of the chip-addition operator allows us to define the action of the set of chip configurations $c$ on  the set of rotor configurations of $G$ by $c(\rho):=\underset{v\in V\backslash \{s\}}{\mathlarger{\prod}} \chipadd{v}^{c(v)} \rho$.
 The following implies a bijective proof for the matrix-tree theorem.
\begin{lemma}\cite{HLMPPW08}
\label{lem:eclassaction}
Let $G$ be a digraph with a global sink $s$, $\rho$ be an acyclic rotor configuration on $G$, and $\sigma_1,\sigma_2$ be two chip configurations of $G$. Then $\sigma_1(\rho)=\sigma_2(\rho)$ if and only if $\sigma_1$ and $ \sigma_2$ are in the same equivalence class.
\end{lemma} 
\section{Orbits of rotor-router operation}
\label{section:orbitsize}
\renewcommand\numedges[3]{a({#2,#3})}
\renewcommand\outdegree[2]{deg^{+}(#2)}
\renewcommand\outdeg[2]{deg^{+}(#2)}
\renewcommand\indeg[2]{deg^{-}(#2)}
\renewcommand\nsp[2]{\mathcal{T}(#2)}
In this section we work with a connected digraph  $G=(V,E)$. For simplicity we use the notations $deg^{+}(v),deg^{-}(v)$ and $a(v,v')$ to stand for $deg_G^{+}(v),deg_G^{-}(v)$ and $a_G(v,v')$, respectively. Fix a linear order $v_1<v_2<\dots<v_n$ on $V$, where $n=|V|$, and let $\Delta$ denote the Laplacian matrix of $G$ with respect to this order. For each vertex $v$ let $\nsp{G}{v}$ denote the number of oriented spanning trees of $G$ rooted at $v$. Let $M$ denote the greatest common divisor of the numbers in $\{\nsp{G}{v}:v\in V\}$. The following lemma is a variant of the Markov chain tree theorem which will be important in the proof Theorem \ref{theo:orbitsize} (see \cite{AT89,LR83}).
% \footnote{This result was mentioned in \cite{PPW11} with a reference to a work which was in progress. However we could not find the result in that work. So we decide to give a proof for this fact.}
\begin{lemma}\label{lem:leftkernel}
$(\nsp{G}{v_1},\nsp{G}{v_2},\dots,\nsp{G}{v_n})\Delta=\textbf{0}$, where $\textbf{0}$ denotes the row vector in $\mathbb{Z}^n$ whose entries are $0$.
\end{lemma}
%\begin{proof}
%Let $D_{i,j}$ denote the matrix that is obtained from $\Delta$ by deleting the $i^{th}$ row and $j^{th}$ column. We claim that $det(D_{i,j})=(-1)^{i+j} \nsp{G}{v_i}$. Clearly, by the matrix-tree theorem the claim holds for $i=j$. So we assume that $i\neq j$. If suffices to show that $det(D_{2,1})=-\nsp{G}{v_2}$ since otherwise we can repeatedly switch between rows and between columns so that we obtain a new Laplacian matrix with respect to an linear order on $V$ in which $v_j$ and $v_i$ are the first and second elements in this order, respectively. Then we continue the proof with this matrix. Let $\Delta'$ denote the matrix obtained from $\Delta$ by deleting the second row and the second column. Since the sum of all columns of $\Delta^{'}$ is equal to minus the first column of $D_{2,1}$, and the other columns of $D_{2,1}$ are the same as those of $\Delta^{'}$, we have $det(\Delta^{'})=-det(D_{2,1})$. By the matrix-tree theorem we have $det(\Delta^{'})=\nsp{G}{v_2}$, therefore $det(D_{2,1})=-\nsp{G}{v_2}$.
%
%Since $det(\Delta)=0$, for any $j\in \{1,2,\dots,n\}$ we have
%\begin{align*}
%0&=\det(\Delta)=\underset{1\leq i \leq n}{\mathlarger{\sum}} (-1)^{i+j} \Delta_{i,j} det(D_{i,j})\\
%&=\underset{1\leq i\leq n}{\mathlarger{\sum}}  \nsp{G}{v_i} \Delta_{i,j}=(\nsp{G}{v_1},\nsp{G}{v_2},\dots,\nsp{G}{v_n}) (\Delta_{1,j},\Delta_{2,j},\dots,\Delta_{n,j})^{\top}
%\end{align*}
%This implies that $(\nsp{G}{v_1},\nsp{G}{v_2},\dots,\nsp{G}{v_n}) \Delta=\textbf{0}$.
%\end{proof}

From now until the end of this section we assume $G$ to be strongly connected. This assumption implies that $\nsp{G}{v}\geq 1$ for any $v\in V$.
\begin{corollary}
\label{coro:generator}
The vector $\frac{1}{M} (\nsp{G}{v_1},\nsp{G}{v_2},\dots,\nsp{G}{v_n})$ is a generator of the kernel of the operator $z\mapsto z \Delta$ in $(\mathbb{Z}^n,+)$. 
\end{corollary}
\begin{proof}
We consider the operator $z\mapsto z \Delta$  in the vector space $\mathbb{Q}^n$ over the field $\mathbb{Q}$. Since $\Delta$ has rank $n-1$, the kernel has dimension $1$ in $\mathbb{Q}^n$.  By Lemma \ref{lem:leftkernel} the vector $(\nsp{G}{v_1},\nsp{G}{v_2},\dots,\nsp{G}{v_n})$ is in the kernel. Thus for any vector $z\in \mathbb{Z}^n$ such that $z \Delta=0$ there exists $q\in \mathbb{Q}$ such that $z=q (\nsp{G}{v_1},\nsp{G}{v_2},\dots,\nsp{G}{v_n})$. Since $M$ is the greatest common divisor of the numbers $\nsp{G}{v_1},\nsp{G}{v_2},\dots,\nsp{G}{v_n}$, we have $q M\in \mathbb{Z}$. This implies that $\frac{1}{M} (\nsp{G}{v_1},\nsp{G}{v_2},\dots,\nsp{G}{v_n})$ is a generator of the kernel of $z\mapsto z \Delta$ in $(\mathbb{Z}^n,+)$.
\end{proof}
%\pagebreak
\begin{figure}
\centering
%\subfloat[$(w_1,\rho_{i_j})=(w_{i_j},\rho_{i_j})$]{\includegraphics[bb=0 0 201 200,width=1.5in,height=1.5in]{im9.eps}}
\subfloat[$(w_1,\rho_{i_j})=(w_{i_j},\rho_{i_j})$]{\includegraphics[bb=0 0 201 200,width=1.5in,height=1.5in]{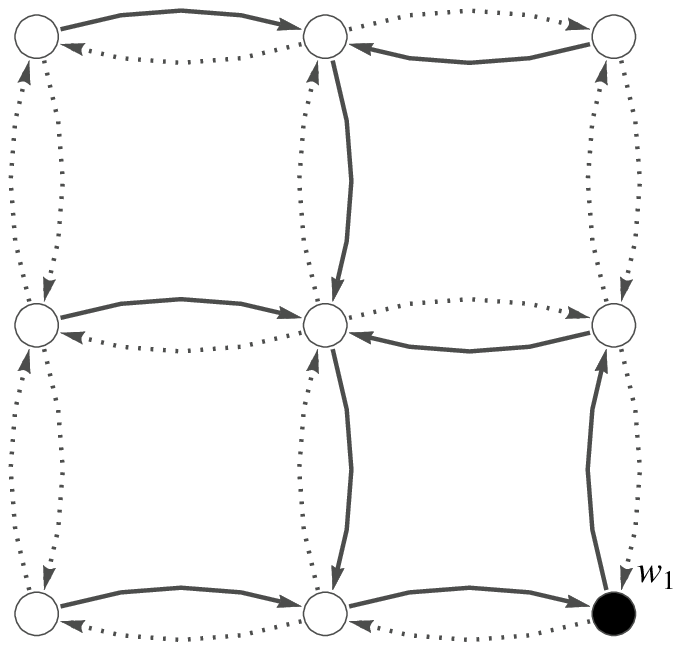}} \quad
%\subfloat[$(w_{i_j+1},\rho_{i_j+1})$]{\includegraphics[bb=0 0 198 197,width=1.5in,height=1.5in]{im10.eps}}
\subfloat[$(w_{i_j+1},\rho_{i_j+1})$]{\includegraphics[bb=0 0 198 197,width=1.5in,height=1.5in]{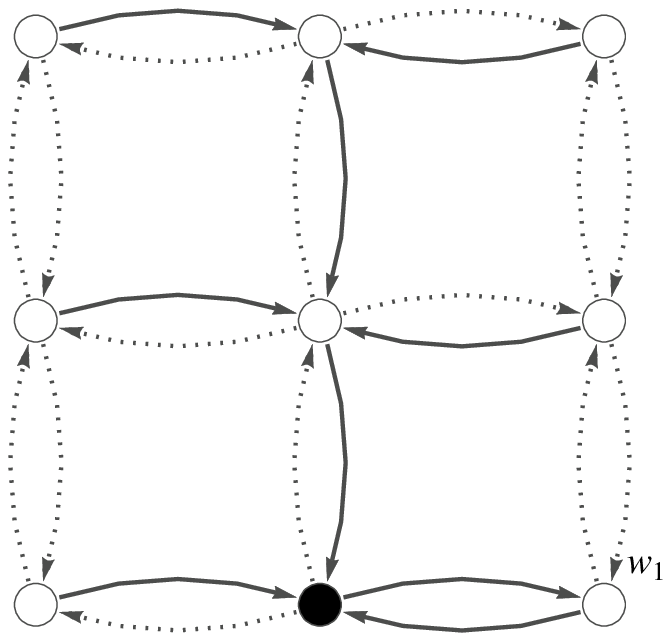}}\quad
%\subfloat[$(w_{i_j+2},\rho_{i_j+2})$]{\includegraphics[bb=0 0 192 190,width=1.5in,height=1.5in]{im11.eps}}
\subfloat[$(w_{i_j+2},\rho_{i_j+2})$]{\includegraphics[bb=0 0 192 190,width=1.5in,height=1.5in]{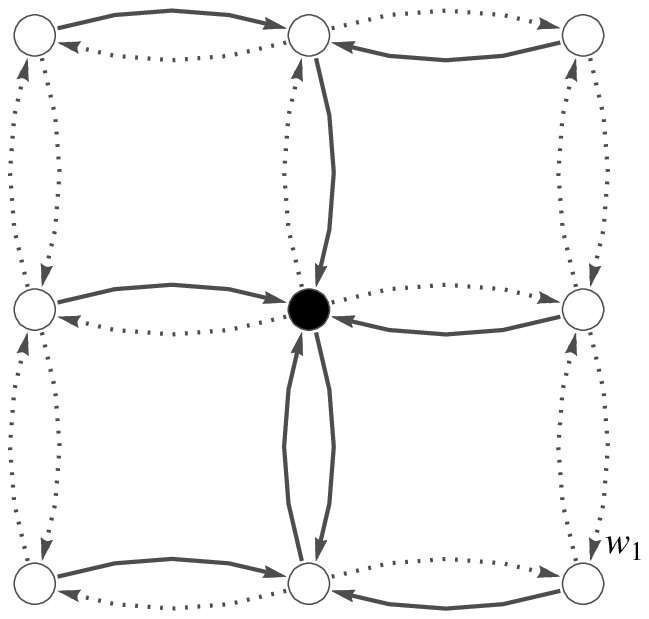}}\\
%\subfloat[$(w_{i_j+3},\rho_{i_j+3})$]{\includegraphics[bb=0 0 186 184,width=1.7in,height=1.5in]{im12.eps}}
\subfloat[$(w_{i_j+3},\rho_{i_j+3})$]{\includegraphics[bb=0 0 186 184,width=1.7in,height=1.5in]{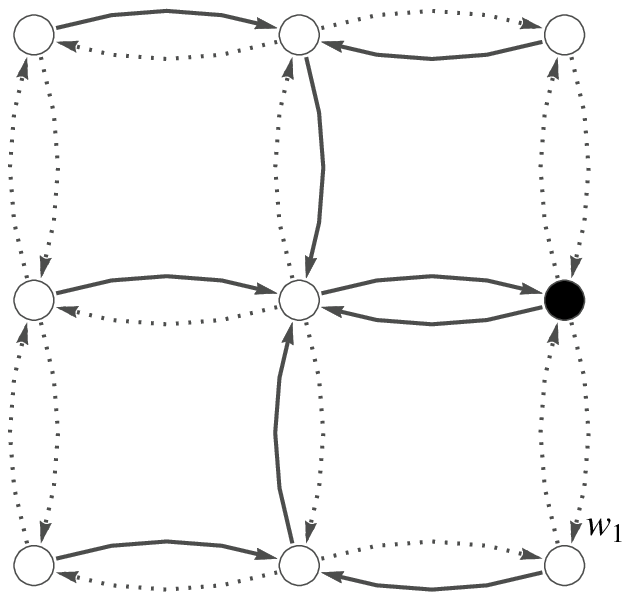}}\quad \quad
%\subfloat[$(w_{i_j+4},\rho_{i_j+4})=(w_1,\rho_{i_{j+1}})$]{\includegraphics[bb=0 0 180 178,width=1.7in,height=1.5in]{im13.eps}}
\subfloat[$(w_{i_j+4},\rho_{i_j+4})=(w_1,\rho_{i_{j+1}})$]{\includegraphics[bb=0 0 180 178,width=1.7in,height=1.5in]{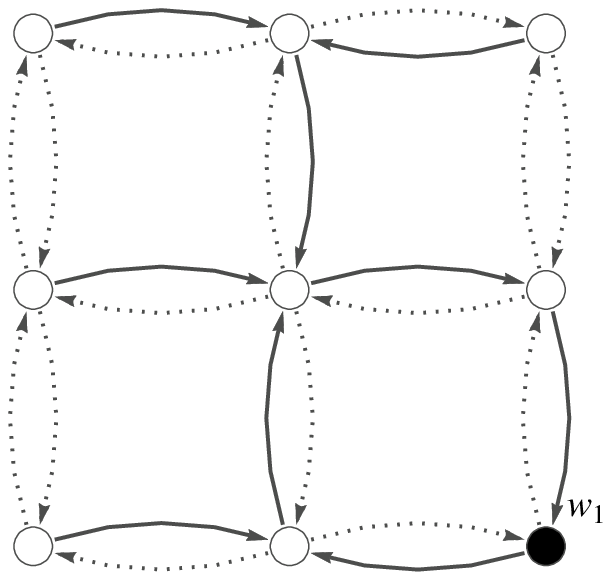}}\\
%\subfloat[$\overline{\rho_{i_j}}$]{\includegraphics[bb=0 0 260 259,width=1.5in,height=1.5in]{im14.eps}}
\subfloat[$\overline{\rho_{i_j}}$]{\includegraphics[bb=0 0 260 259,width=1.5in,height=1.5in]{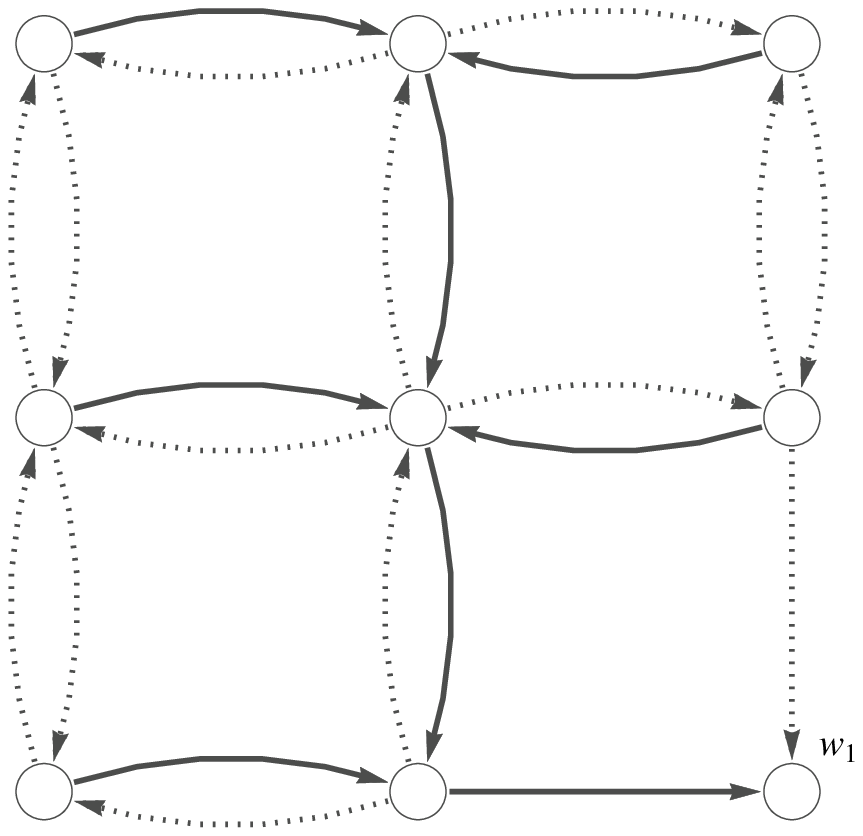}}
\quad \quad
%\subfloat[$\overline{\rho_{i_{j+1}}}$]{\includegraphics[bb=0 0 260 259,width=1.5in,height=1.5in]{im15.eps}}
\subfloat[$\overline{\rho_{i_{j+1}}}$]{\includegraphics[bb=0 0 260 259,width=1.5in,height=1.5in]{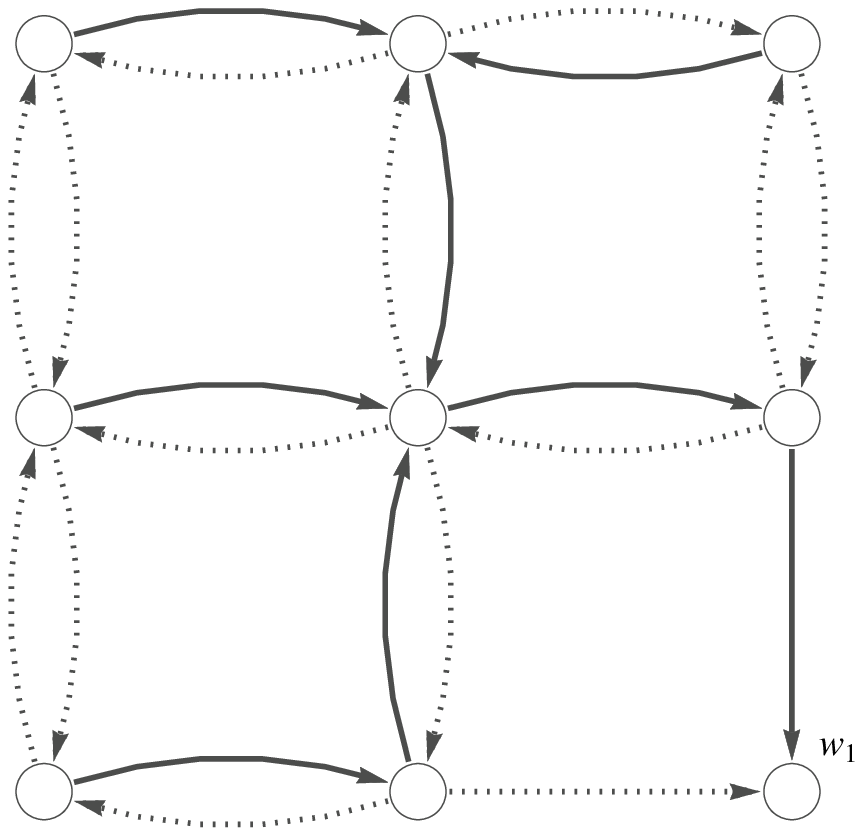}}
\caption{}
\label{fig:im0910111213}
\end{figure}
\begin{lemma}
\label{lem:spantoorder}
For $i\in \{1,2,\dots,n\}$ let $\Delta^{'}$ denote the matrix obtained from $\Delta$ by deleting the $i^{th}$ column. Then the order of $\Delta_i^{'}$ in the quotient group $(\mathbb{Z}^{n-1},+)/\langle\{\Delta_{j}^{'}:j\neq i\}\rangle$ is $\frac{\nsp{G}{v_i}}{M}$.
\end{lemma} 
\begin{proof}
Clearly, the order of $\Delta_i^{'}$ in $(\mathbb{Z}^{n-1},+)/\langle\{\Delta_{j}^{'}:j\neq i\}\rangle$ is the smallest positive integer $p_i$ such that there exist integers $p_1,p_2,\dots,p_{i-1},p_{i+1},\dots,p_n$ such that $p_i \Delta_i^{'}=\underset{j\neq i}{\sum} p_j \Delta_j^{'}$, equivalently $$(-p_1,-p_2,\dots,-p_{i-1},p_i,-p_{i+1},\dots,-p_n)\Delta=\textbf{0}.$$

\noindent It follows from Corollary \ref{coro:generator} that $p_i=\frac{\nsp{G}{v_i}}{M}$.
\end{proof}
\begin{proof}[Proof of Theorem \ref{theo:orbitsize}]
Let $(w_1,\rho_1)$ be an arbitrary unicycle of $G$. Let $(w_1,\rho_1),(w_2,\rho_2),(w_3,\rho_3),\dots$ be the infinite sequence of states such that for any $i\geq 1$ the state $(w_{i+1},\rho_{i+1})$ is obtained from the state $(w_i,\rho_i)$ by applying the rotor-router operation. By collecting all states $(w_i,\rho_i)$ with $w_i=w_1$ we obtain the subsequence $(w_1,\rho_{i_1}),(w_1,\rho_{i_2}),(w_1,\rho_{i_3}),\dots$. Note that $1=i_1$. For each $\rho_{i_j}$ let $u_j$ denote the head of $\rho_{i_j}(w_1)$. Let $e_1,e_2,\dots,e_k$, where $k=\outdegree{G}{w_1}$, be an enumeration of the edges emanating from $w_1$ such that $e_1=\rho_1(w_1)$ and $e_{i+1}=e_i^{+}$ for any $i<k$, and $e_1=e_k^{+}$.

Let $\overline{G}$ denote the graph obtained from $G$ by deleting all  edges emanating from $w_1$, and for each $\rho_{i_j}$ let $\overline{\rho_{i_j}}$ denote the restriction of $\rho_{i_j}$ on $\overline{G}$. We have that $\overline{\rho_{i_j}}$ is an acyclic rotor configuration of $\overline{G}$ (see Figure \ref{fig:im0910111213}). It follows from the definition of the chip addition operator that $\overline{\rho_{i_{j+1}}}=\chipadd{u_{j+1}} \overline{\rho_{i_j}}$. Note that if $u_{j+1}=w_1$, then $\overline{\rho_{i_{j+1}}}=\overline{\rho_{i_j}}$. For each $q>1$ we define the chip configuration $c_q:V\backslash \{w_1\}\to\mathbb{N}$ by for any $v\in V\backslash \{w_1\}$ $c_q(v)$ is the number of occurrences of $v$ in the sequence $u_2,u_3,\dots,u_q$. The above identity implies that $\overline{\rho_{i_q}}=c_q(\overline{\rho_{i_1}})$. Let $\Delta^{'}$ be the matrix that is obtained from $\Delta$ by deleting the column corresponding to $w_1$. We have $\rho_{i_q}=\rho_{i_1}$  if and only if the following conditions hold
\begin{itemize}
  \item[-] the configuration $c_q$ is in the same equivalence class as $\textbf{0}$ in $\overline{G}$. This fact follows from \linebreak Lemma \ref{lem:eclassaction}.
  \item[-] $c_q=-p \Delta'_{w_1}$ for some $p$, where $\Delta'_{w_1}$ denotes the row of $\Delta'$ corresponding to the vertex $w_1$. This follows the fact that the sequence $\rho_{i_1}(w_1),\rho_{i_2}(w_1),\rho_{i_3}(w_1)\dots$ is exactly the periodic sequence $e_1,e_2,\dots,e_k,e_1,e_2,\dots,e_k,\dots$ Note that $\rho_{i_2}(w_1),\rho_{i_3}(w_1),\dots,\rho_{i_q}(w_1)$ is a periodic sequence of length $p k$, namely $\underset{\text{length } pk}{\underbrace{e_2,e_3,\dots,e_k,e_1,\dots,e_2,e_3\dots,e_k,e_1}}$.
\end{itemize} 
Thus $1+pk$ is the smallest $q$ satisfying $\rho_{i_1}=\rho_{i_q}$, where $p$ is the order of $\Delta'_{w_1}$ in $\mathbb{Z}^{n-1}/\langle\{\Delta'_{v}:v\in V\backslash\{ w_1\}\}\rangle$. By Lemma \ref{lem:spantoorder} we have $p=\frac{1}{M}\nsp{G}{w_1}$. It follows that in the orbit $\{(w_i,\rho_i):1\leq i\leq i_{1+p k}-1\}$ the number of times the chip passes through $w_1$ is $\frac{1}{M}\outdeg{G}{w_1}\nsp{G}{w_1}$. Since this fact also holds for other vertices, the size of orbit is $\frac{1}{M} \underset{v\in V}{\mathlarger{\sum}} \outdegree{G}{v}\nsp{G}{v}$.

Since the number of unicycles is $\underset{v\in V}{\mathlarger{\sum}}\outdegree{G}{v} \nsp{G}{v}$, it follows that the number of orbits of the rotor-router operation is $M$.
\end{proof}

If $G$ is an Eulerian digraph, then the numbers of oriented spanning trees $\nsp{G}{v},v\in V$ are the same since $\nsp{G}{v}$ is equal to the order of the sandpile group of $G$ with sink $v$ and the sandpile group is independent of the choice of sink \cite{HLMPPW08}. Thus $M=\nsp{G}{v_1}=\nsp{G}{v_2}=\dots=\nsp{G}{v_n}$. By Theorem \ref{theo:orbitsize} each orbit of the rotor-router operation has size $\underset{v\in V}{\mathlarger{\sum}} \outdegree{G}{v}=|E|$. We recover the result in \cite{HLMPPW08,PDDS96}.
\begin{proposition}\cite{HLMPPW08,PDDS96}
Let $G$ be an Eulerian digraph with $m$ edges. Starting from a unicycle $(w,\rho)$ the chip traverses each edge exactly once before returning to $(w,\rho)$ for the first time. 
\end{proposition}
\renewcommand\outdeg[2]{deg_{#1}^{+}(#2)}
\renewcommand\indeg[2]{deg_{#1}^{-}(#2)}
\renewcommand\nsp[2]{\mathcal{T}_{#1}(#2)}
\begin{proof}[Proof of Theorem \ref{theo:singleorbit}]
For each $n\geq 3$ let $G_n$ be the strongly connected digraph given by $V(G_n):=\{1,2,\dots,n\}$ and $E(G_n):=\{ (i,i+1):1\leq i \leq n-1 \}\cup \{ (i,1):2\leq i \leq n \}$. Since $\outdeg{G_n}{1}=1$ and $\indeg{G_n}{1}=n-1$, $G_n$ is not Eulerian. Since $G_n$ has exactly one oriented spanning tree rooted at $n$, namely the subgraph induced by the edges in $\{ (i,i+1):1\leq i\leq n-1 \}$, we have $\nsp{G_n}{n}=1$, therefore $M_{G_n}=1$. By Theorem \ref{theo:orbitsize} all unicycles are in a single orbit.
\end{proof}

\def\newstd{\bar{\pi}}
\renewcommand\outdegree[2]{deg_{#1}^{+}(#2)}
\renewcommand\indegree[2]{deg_{#1}^{-}(#2)}
\renewcommand\nsp[2]{\mathcal{T}_{#1}(#2)}

The formula in Theorem \ref{theo:orbitsize} is very useful because one can use it to compute size of an orbit effeciently without listing all unicycles in an orbit. As we saw above, size of orbits on a strongly connected digraph is often large while it is extremely short on an Eulerian digraph. If orbit size is too large (resp. too small), then number of orbits is too small (resp. too large). Thus one would expect to see an infinite family of strongly connected digraphs $G_n$ on which the rotor-router operation behaves moderately, i.e. both the orbit size and the number of orbits grow exponentially in the number of vertices and in the number of edges. By using Theorem \ref{theo:orbitsize} we construct easily such a family of digraphs as follows. For $n\geq 1$ the graph $G_n$ has the vertex set $\{1,2,\dots, n+1\}$, and for each $i\in \{1,2,\dots,n\}$ there are two edges connecting $i$ to $i+1$ and four edges connecting $i+1$ to $i$ in $G_n$. It is easy to see that $\nsp{G_n}{i}=4^{n+1-i}\times 2^{i-1}=2^{2 n+1-i}$ for any $i\in \{1,2,\dots,n+1\}$. Therefore we have $M_{G_n}=2^n$. It follows from Theorem \ref{theo:orbitsize} that the number of orbits is $2^n$ and the size of orbits  is greater than $\frac{\nsp{G_n}{1}}{2^n}=2^n$. Thus the family of digraphs $G_n$ has the desired property.\\

\noindent{\large\textbf{Acknowledgements.}} We are thankful to L. Levine, M. Farrell and the referee for their useful comments and suggestions on the paper.

\text{}\\
\text{}\\
Technische Universit\"{a}t Dresden\\
Fachrichtung Mathematik, Institut f\"{u}r Algebra\\
01062 Dresden, Germany\\
\text{}\\
Institute of Mathematics, VAST\\
Department of Mathematics for Computer Science\\
18 Hoang Quoc Viet Road, Cau Giay District, Hanoi, Vietnam.\\
\text{}\\
Email address: pvtrung@math.ac.vn
\end{document}